\documentclass[11pt]{article}

\usepackage[a4paper,top=2.35cm,bottom=2.5cm,left=2.45cm,right=2.45cm,
            headheight=22pt,headsep=0.55cm]{geometry}
\usepackage[T1]{fontenc}
\usepackage[utf8]{inputenc}
\usepackage{amsmath,amssymb,amsthm,mathtools}
\usepackage{newtxtext}
\usepackage{newtxmath}
\usepackage[scaled=0.94]{sourcesanspro}
\usepackage{microtype}
\usepackage{bm}
\usepackage{mathrsfs}
\usepackage{graphicx}
\usepackage{subcaption}
\usepackage{booktabs}
\usepackage{array}
\usepackage{tabularx}
\usepackage{longtable}
\usepackage{enumitem}
\usepackage{xcolor}
\usepackage{titlesec}
\usepackage{fancyhdr}
\usepackage[most]{tcolorbox}
\usepackage{xurl}
\usepackage{hyperref}
\usepackage[nameinlink,noabbrev]{cleveref}
\usepackage[numbers,sort&compress]{natbib}
\usepackage{tikz}
\usetikzlibrary{arrows.meta,positioning,fit}

\definecolor{PaperNavy}{HTML}{17324D}
\definecolor{PaperTeal}{HTML}{1E6F72}
\definecolor{PaperInk}{HTML}{22272B}
\definecolor{PaperMuted}{HTML}{66727C}
\definecolor{PaperLine}{HTML}{D9E0E4}
\definecolor{PaperWash}{HTML}{F4F7F8}
\definecolor{PaperWarm}{HTML}{8A5A20}
\definecolor{PaperWarmWash}{HTML}{FBF7F1}

\color{PaperInk}
\setlist{nosep,leftmargin=2.0em,topsep=0.45em}
\setlist[enumerate,1]{label=\textcolor{PaperTeal}{\arabic*.}}
\setlist[itemize,1]{label=\textcolor{PaperTeal}{\raisebox{0.2ex}{\small$\bullet$}}}

\titleformat{\section}
  {\Large\sffamily\bfseries\color{PaperNavy}}
  {\textcolor{PaperTeal}{\thesection}}{0.72em}{}
  [\vspace{0.25em}\color{PaperLine}\titlerule]
\titleformat{\subsection}
  {\large\sffamily\bfseries\color{PaperNavy}}
  {\textcolor{PaperTeal}{\thesubsection}}{0.68em}{}
\titleformat{\subsubsection}
  {\normalsize\sffamily\bfseries\color{PaperNavy}}
  {\textcolor{PaperTeal}{\thesubsubsection}}{0.65em}{}
\titlespacing*{\section}{0pt}{2.0em}{0.8em}
\titlespacing*{\subsection}{0pt}{1.5em}{0.45em}
\titlespacing*{\subsubsection}{0pt}{1.1em}{0.35em}

\hypersetup{
  colorlinks=true,
  linkcolor=PaperTeal,
  citecolor=PaperTeal,
  urlcolor=PaperTeal,
  pdfborder={0 0 0},
  pdftitle={Scalable Persistence Pairing Graphs},
  pdfauthor={John Rick Manzanares}
}
\renewcommand{\headrulewidth}{0.35pt}
\renewcommand{\headrule}{\hbox to\headwidth{\color{PaperLine}\leaders\hrule height \headrulewidth\hfill}}

\newtheoremstyle{paperdefinition}
  {0.75em}{0.75em}{\normalfont}{}%
  {\sffamily\bfseries\color{PaperNavy}}{.}{0.55em}{}
\newtheoremstyle{paperplain}
  {0.75em}{0.75em}{\itshape}{}%
  {\sffamily\bfseries\color{PaperNavy}}{.}{0.55em}{}
\newtheoremstyle{paperremark}
  {0.65em}{0.65em}{\normalfont}{}%
  {\sffamily\bfseries\color{PaperNavy}}{.}{0.55em}{}

\theoremstyle{paperdefinition}
\newtheorem{definition}{Definition}[section]
\theoremstyle{paperplain}
\newtheorem{theorem}[definition]{Theorem}
\newtheorem{proposition}[definition]{Proposition}
\newtheorem{lemma}[definition]{Lemma}
\newtheorem{corollary}[definition]{Corollary}
\theoremstyle{paperremark}

\tcolorboxenvironment{definition}{
  enhanced,breakable,colback=PaperWash,colframe=PaperLine,
  boxrule=0.45pt,arc=1.2pt,left=9pt,right=9pt,top=6pt,bottom=6pt,
  borderline west={2pt}{0pt}{PaperTeal}
}
\tcolorboxenvironment{theorem}{
  enhanced,breakable,colback=white,colframe=PaperLine,
  boxrule=0.45pt,arc=1.2pt,left=9pt,right=9pt,top=6pt,bottom=6pt,
  borderline west={2pt}{0pt}{PaperNavy}
}
\tcolorboxenvironment{proposition}{
  enhanced,breakable,colback=white,colframe=PaperLine,
  boxrule=0.45pt,arc=1.2pt,left=9pt,right=9pt,top=6pt,bottom=6pt,
  borderline west={2pt}{0pt}{PaperNavy}
}
\tcolorboxenvironment{lemma}{
  enhanced,breakable,colback=white,colframe=PaperLine,
  boxrule=0.45pt,arc=1.2pt,left=9pt,right=9pt,top=6pt,bottom=6pt,
  borderline west={2pt}{0pt}{PaperNavy}
}
\tcolorboxenvironment{corollary}{
  enhanced,breakable,colback=white,colframe=PaperLine,
  boxrule=0.45pt,arc=1.2pt,left=9pt,right=9pt,top=6pt,bottom=6pt,
  borderline west={2pt}{0pt}{PaperNavy}
}

\newcommand{\Ftwo}{\mathbb F_2}
\newcommand{\R}{\mathbb R}

\newcommand{\low}{\operatorname{low}}
\newcommand{\im}{\operatorname{im}}
\newcommand{\coker}{\operatorname{coker}}
\newcommand{\supp}{\operatorname{supp}}

\newcommand{\Dying}{\mathcal D}
\newcommand{\Surv}{\mathcal S}

\makeatletter
\title{Scalable persistence pairing graphs on scalar fields with an application to porous media}
\author{John Rick Manzanares\textsuperscript{1,2,*}}
\date{7 October 2026}
\newcommand{\papertype}{Preprint draft}
\newcommand{\affiliations}{%
\begin{tabularx}{\linewidth}{@{}>{\raggedleft\arraybackslash\sffamily\bfseries\color{PaperTeal}}p{1.5em} X@{}}
1 & Institute of Mathematics of the Polish Academy of Sciences, 00-656 Warsaw, Poland\\[0.2em]
2 & International Environmental Doctoral School, University of Silesia in Katowice, 41-200 Sosnowiec, Poland
\end{tabularx}}
\newcommand{\correspondence}{John Rick Manzanares, \href{mailto:jdolormanzanares@impan.pl}{jdolormanzanares@impan.pl}}
\newcommand{\funding}{This project has received funding from the European Union's Horizon Europe research and innovation programme under the Marie Sk\l{}odowska-Curie Actions (MSCA), grant agreement No.~101120290 (GAP).}
\newcommand{\competinginterests}{The authors declare no competing interests.}
\newcommand{\keywords}{topological data analysis; persistent homology;
digital rock physics; signed distance transform; homological relations
}

\renewcommand{\maketitle}{%
  \thispagestyle{empty}
  \noindent
  \begin{minipage}[t]{0.69\textwidth}
    {\sffamily\footnotesize\bfseries\color{PaperTeal}\MakeUppercase{\papertype}}\\[0.75em]
    {\raggedright\fontsize{23}{27}\selectfont\sffamily\bfseries\color{PaperNavy}\@title\par}
    \vspace{0.85em}
    {\large\sffamily\color{PaperInk}\@author\par}
  \end{minipage}%
  \hfill
  \begin{minipage}[t]{0.24\textwidth}
    \raggedleft
    {\sffamily\footnotesize\color{PaperMuted}Version date\\[0.2em]
    \color{PaperInk}\@date}
  \end{minipage}

  \vspace{1.0em}
  \noindent{\color{PaperTeal}\rule{\textwidth}{1.15pt}}
  \vspace{0.95em}

  {\small\noindent\affiliations\par}
  \vspace{0.75em}

  \noindent\begin{tcolorbox}[
    enhanced,colback=PaperWash,colframe=PaperLine,boxrule=0.4pt,
    arc=1.2pt,left=10pt,right=10pt,top=8pt,bottom=8pt]
  \noindent\begin{tabularx}{\linewidth}{@{}>{\raggedright\arraybackslash\sffamily\bfseries\color{PaperNavy}}p{3.7cm} >{\raggedright\arraybackslash}X@{}}
  Correspondence & \correspondence\\[0.32em]
  Funding & \funding\\[0.32em]
  Competing interests & \competinginterests\\[0.32em]
  Keywords & \keywords
  \end{tabularx}
  \end{tcolorbox}
  \vspace{0.35em}
}
\makeatother

\begin{document}
\maketitle

\begin{abstract}
\noindent
We present a scalable method for computing persistence pairing graphs from large cubical filtrations and use it to test whether graph organization among persistent classes contains information beyond persistence intervals alone. The method cancels equal-value pairs before global reduction, constructs the compressed complex in streaming form, and recovers graph transitions from one persistence reduction through projected relations and triangular event solves. We apply the method to 119 signed distance volumes from the DRP-372 porous media collection for permeability prediction. Each 256-cubed volume induces more than 135 million cubical cells before contraction, whereas the median compressed complex contains 37,471 generators. In validation grouped by material family, adding persistence pairing graph descriptors improves prediction over persistence summaries under both Ridge regression and partial least squares regression. In the stricter experiment that holds out one source project at a time, Ridge retains a modest improvement over persistence alone, although additional benefit beyond geometry and persistence is not consistent across projects. Two null models show that the observed persistence and graph coupling and edge target organization are strongly nonrandom, but neither uniquely explains the predictive improvement; the useful signal appears to reside mainly in coarser graph and event organization. Spatial analyses likewise distinguish persistence magnitude from graph prominence and reveal nonlocal organization among algebraically related persistence events. These results show that persistence pairing graphs provide structurally informative summaries beyond the barcode while exposing domain dependence and representative sensitivity that motivate more invariant relational descriptors.
\end{abstract}

\section{Introduction}

Persistent homology summarizes how topological features appear and disappear across a filtration. In its standard interval representation, a class is described by its birth and death parameters, producing a persistence barcode or persistence diagram \cite{edelsbrunner2002topological,zomorodian2005computing,crawleyboevey2015decomposition}. This representation is deliberately economical. It records lifetime but generally discards how simultaneously active reduction-selected homology classes are related when one of them dies.

A persistence pairing graph (PPG) retains one additional layer. Its vertices are nonzero persistence intervals, while a directed edge records a nonzero transition coefficient from a dying selected class to a selected class that survives the same filtration event. The companion theory paper develops this construction, its algebraic interpretation, and its representative dependence \cite{manzanares2026PPG}. In higher homological dimensions, a PPG edge is a basis-transition relation; it should not be interpreted automatically as a geometric merger or as a physical channel in the underlying domain. The present paper addresses two questions left open by the theory: whether these relations can be computed at large cubical-image scale, and whether they carry measurable information in a physical application.

The computational problem differs from barcode-only persistence. High-performance persistence software can discard much of the change-of-basis information once persistence pairs are known. PHAT, Ripser, Cubical Ripser, distributed cubical approaches, GUDHI, and recent cubical methods exploit this freedom in different ways \cite{bauer2017phat,bauer2021ripser,kaji2020cubicalripser,bauer2014distributed,maria2014gudhi,lebreton2026flash}. PPG extraction instead requires enough reduction history to recover selected basis transitions after persistence pairing. Representative-aware factorizations such as U-match similarly illustrate that the storage requirements change when one wants more than a barcode \cite{hang2021umatch}.

Three-dimensional scalar images make this distinction costly. In the doubled-grid cubical representation used here, an $n_1\times n_2\times n_3$ array induces
\[
(2n_1+1)(2n_2+1)(2n_3+1)
\]
cubical cells. A $256^3$ scalar volume therefore corresponds to $513^3=135{,}005{,}697$ cells before compression. At this scale, materializing the full cubical complex, all boundaries, and a full change-of-basis matrix is not a viable route to graph extraction.

Our computational strategy uses a property of discrete image filtrations that is usually viewed as a nuisance: many cells share exactly the same filtration value. A deterministic cellwise refinement of these ties creates large numbers of zero-length persistence pairs. Those pairs are absent from the positive-length PPG, yet a direct global reduction may still spend most of its work constructing and propagating their algebra. We therefore remove same value algebra first. The cancellation is filtered, so the persistence module is preserved; after this step, a much smaller filtered complex remains for the global persistence and PPG computation.

The optimized implementation is built around this separation. Cubical boundaries are generated on demand. Same-value blocks are contracted in a deterministic order. The compressed complex uses compact filtered-cell records and flat $32$-bit sparse arrays. The global persistence reduction stores only the reduced information and elementary operations required later. Transition systems are formed only on labels that actually occur in a death event. Common short relations are encoded directly, while longer relations share flat arenas. Finally, the discrete filtration order is stored in one delta-coded byte arena whose already-consumed pages can be reclaimed on Linux. These changes reduce the peak resident memory of our full benchmark by more than an order of magnitude while preserving the deterministic value block graph exactly in end-to-end regression tests.

Scalability matters here because the scientific test requires many large volumes rather than one demonstration image. We use the public DRP-372 collection of digital porous media \cite{santos2022drp372}. Previous work has shown that persistent homology and related topological constructions can characterize pore connectivity, percolating length scales, pore heterogeneity, capillary trapping, and topology-preserving pore network structure \cite{robins2016porous,jiang2018pore,herring2019trapping,zubov2022pore}. The question is narrower: after persistence intervals have already been summarized, does the relational information retained by the PPG improve held out prediction of permeability?

We design the application as an ablation study. Geometry, persistence intervals, and graph relations are separated into nested feature sets. Validation is grouped first by a curated material/source family and then, more stringently, by complete source project. This hierarchy is important because the dataset contains related samples and augmented geometries; random specimen-level splitting would risk measuring source similarity rather than transferable relationships between structure and properties. We also treat later coefficient, domain shift, and spatial localization analyses as descriptive interpretation rather than as additional feature selection rounds.

The contributions of this paper are as follows:
\begin{enumerate}
  \item We give a filtered same value cancellation argument for the exact tie removal step used by the computation and derive the generator count identity $2N_{\mathrm{fin}}+N_{\mathrm{ess}}$ once the differential strictly lowers filtration value.
  \item We give a deterministic value block construction and a streaming cubical implementation that produces the same compressed chain complex as the materialized block algorithm under the same order and pivot rule.
  \item We recover the persistence pairing graph from one compact persistence reduction. At each death value, the dying coordinate block is unit lower triangular, so transition coefficients are obtained by forward substitution rather than a general matrix inversion.
  \item We compute PPGs for a cohort of $119$ DRP-372 $256^3$ porous media volumes and evaluate geometry, persistence, and PPG feature sets under nested family grouped validation and a stricter leave one project out stress test.
  \item We show that coarse PPG descriptors improve persistence only models, while the improvement beyond geometry plus persistence is domain dependent. Two null models further show that the exact persistence to graph assignment and the exact surviving targets are strongly structured but are not uniquely responsible for the predictive gain. Spatial analyses then examine how graph prominence relates to the locations of persistence critical cells.
\end{enumerate}

Several ideas used by the implementation have close relatives in earlier persistent homology computation. Standard matrix reduction and persistence pairing go back to the foundational algorithms \cite{edelsbrunner2002topological,zomorodian2005computing}. Clearing avoids reductions known in advance to be unnecessary \cite{chenkerber2011twist}. PHAT separates reduction strategies from matrix data structures \cite{bauer2017phat}. Ripser avoids explicit storage of large filtration matrices and exploits ordering structure aggressively \cite{bauer2021ripser}. Cubical-specific methods avoid triangulation and use implicit or compact incidence calculations \cite{wagner2012cubical,kaji2020cubicalripser}. Distributed persistence addresses memory limits across machines \cite{bauer2014distributed}, while discrete Morse and algebraic reductions remove acyclic structure before persistence computation \cite{robins2011morse,dlotko2012morse}. U-match is particularly relevant to the present setting because it is designed to support later representative recovery from a compressed factorization \cite{hang2021umatch}.

Flash Cubical, introduced during the development of this project, is a recent high performance method for barcode computation on regular 2D and 3D cubical V-filtrations over $\Ftwo$. It combines duality, union--find, local pruning, and lookup tables \cite{lebreton2026flash}. Its computational target differs from ours: methods that compute only barcodes can discard information that PPG transition recovery still needs. We therefore do not interpret the benchmark progression in this paper as a speed comparison against Flash Cubical, Cubical Ripser, or other barcode-only libraries.

The same value cancellation argument is likewise not presented as the invention of algebraic cancellation. Its role is to prove filtered exactness for the specific tie-removal step used by the deterministic graph pipeline. On-demand boundaries, compact integers, bitsets, CSR-like arrays, flat arenas, delta coding, and page reclamation are established algorithm-engineering techniques. The computational contribution is their integration around an unusual output requirement: remove as much zero-length work as possible while retaining enough exact reduction information to recover the selected PPG relations.

Persistent homology has a substantial history in porous media analysis. Robins et al.\ related topological persistence of Euclidean-distance filtrations to percolating length scales and pore connectivity in micro-CT images \cite{robins2016porous}. Jiang et al.\ used persistence diagrams to characterize pore heterogeneity \cite{jiang2018pore}. Herring et al.\ related persistent topology to microstructure and capillary trapping in sandstones \cite{herring2019trapping}. Discrete Morse constructions have also been used to extract pore networks while explicitly targeting preservation of pore-space topology \cite{zubov2022pore}. Our application differs by asking whether graph relations among persistent classes add predictive and spatial information beyond persistence summaries themselves.

\section{Preliminaries}
\label{sec:background}

\subsection{Filtered Chain Complex}

We work over the field $\Ftwo=\{0,1\}$, where addition and subtraction are both addition modulo two. This choice removes orientation signs from the implementation and turns addition of sparse chains into a symmetric difference of their supports.

\begin{definition}
A \emph{finite chain complex over $\Ftwo$} is a sequence of finite-dimensional vector spaces and linear maps
\[
\cdots\longrightarrow C_{q+1} \xrightarrow{\partial_{q+1}} C_q \xrightarrow{\partial_q} C_{q-1} \longrightarrow\cdots
\]
such that $\partial_q\partial_{q+1}=0$ for every $q$.
\end{definition}

The elements of $C_q$ are called \emph{$q$-chains}. A chain $z\in C_q$ is a \emph{cycle} when $\partial_qz=0$. A chain is a \emph{boundary} when it equals $\partial_{q+1}c$ for some $c\in C_{q+1}$. Because $\partial^2=0$, every boundary is a cycle. The \emph{quotient vector space}
\[
H_q(C):=\ker\partial_q/\im\partial_{q+1}
\]
is the $q$th homology of the chain complex. These standard chain complex and homology definitions follow, for example, \cite[Chapter~2]{hatcher2002algebraic}.

\begin{definition}
A \emph{filtered chain complex} is a finite chain complex $C_*$ together with a chosen basis $\mathcal E$ and a map $f:\mathcal E\to\R$ such that every basis element appearing with nonzero coefficient in $\partial e$ has filtration value at most $f(e)$.
\end{definition}

For $t\in\R$, let $C_*^t$ be the span of all basis elements $e$ with $f(e)\le t$. The filtration condition implies
\[
\partial(C_q^t)\subseteq C_{q-1}^t,
\]
so $C_*^t$ is itself a chain complex. If $s\le t$, inclusion gives a linear map
\[
H_q(C^s)\longrightarrow H_q(C^t).
\]
The collection of these homology vector spaces and maps is the persistence module of the filtered chain complex \cite{zomorodian2005computing,edelsbrunner_harer_2010}. Over a field, a finite one-parameter persistence module decomposes into interval modules, giving the usual persistence barcode \cite{crawleyboevey2015decomposition}.

For a filtration value $t$, we write $C_*^{t^-}$ for the subcomplex generated by basis elements with filtration value strictly less than $t$.

\subsection{Filtered Cubical Complex}

Consider a $d$-dimensional scalar array (or scalar field)
\[
A:\{0,\ldots,n_1-1\}\times\cdots\times\{0,\ldots,n_d-1\}\to \R.
\]
If the scalar values are stored in a $k$-bit integer image, the codomain is $\left\{0,\ldots,2^k-1\right\}$.

The implementation uses a doubled regular grid with side lengths
\[
2n_1+1,\ldots,2n_d+1.
\]
A grid position with exactly $q$ odd coordinates represents a $q$-dimensional cubical cell. Thus in three dimensions, positions with zero, one, two, or three odd coordinates represent vertices, edges, squares, or voxels. The total number of cells is
\begin{equation}
\label{eq:cubical-cell-count}
N=\prod_{r=1}^d(2n_r+1).
\end{equation}

A cell has two faces along each odd coordinate. In the row-major one-dimensional grid index, these faces are obtained by subtracting or adding the stride of that coordinate. Hence, a $q$-cell has exactly $2q$ codimension-one faces before boundary clipping, and in this doubled representation the valid cell locations make these face steps explicit.

The scalar values are given on top-dimensional cells. For a lower-dimensional cell $\sigma$, we set
\begin{equation}
\label{eq:top-cell-extension}
f(\sigma)
:=
\min\{A(v): \sigma\subseteq\overline v,\ v\text{ a top-dimensional cell}\}.
\end{equation}
This rule guarantees that every face enters no later than its cofaces. It therefore defines a filtered cubical chain complex. Cubical persistent homology avoids the extra cells introduced by triangulating every voxel and is standard for image data \cite{wagner2012cubical,kaji2020cubicalripser}.

\subsection{Persistence Pairing Graph}

Let $c(\sigma)=(c_1,\ldots,c_d)$ be the doubled grid coordinate of a cell $\sigma$, and write $m_r=2n_r+1$ for the side length of the doubled grid in coordinate $r$. We define
\begin{equation}
\label{eq:grid-index}
\operatorname{gridindex}(\sigma)
:=
\sum_{r=1}^{d} c_r\prod_{s=r+1}^{d}m_s,
\end{equation}
where an empty product equals one. Thus $\operatorname{gridindex}(\sigma)$ is the zero-based row-major index of the doubled grid position, with the last coordinate varying fastest.

Order all cells by the lexicographic key
\begin{equation}
\label{eq:cell-order}
\bigl(f(\sigma),\ \dim\sigma,\ \operatorname{gridindex}(\sigma)\bigr).
\end{equation}
The dimension tie-break ensures that a face appears before a coface when both have the same scalar value. Let
\[
\sigma_1\prec\sigma_2\prec\cdots\prec\sigma_N
\]
be this total order.

The boundary matrix $D$ has one row and one column for every cell. Column $j$ contains the boundary of $\sigma_j$. Because every face is earlier than its coface, every nonzero entry of column $j$ lies in a row $i<j$.

For a nonzero column $c$, we define
\[
\low(c):=\max\supp(c),
\]
as the largest row index whose coefficient is one. Standard persistence reduction scans columns from left to right. Whenever two nonzero columns $R_i$ and $R_j$ have the same lowest row and $i<j$, replace
\[
R_j\leftarrow R_j+R_i.
\]
Over $\Ftwo$, this is a bitwise XOR of the supports. Reduction stops when every nonzero column has a distinct lowest row. A zero reduced column is called \emph{positive}; it creates a homology class. A nonzero reduced column is called \emph{negative}; if $\low(R_j)=i$, then positive cell $i$ and negative cell $j$ form a persistence pair \cite{edelsbrunner2002topological,zomorodian2005computing,edelsbrunner_harer_2010}.

If the same column additions are also applied to an identity matrix, then
\begin{equation}
\label{eq:rdv}
R=DV,
\end{equation}
where $V$ is upper unitriangular. The positive column $V_i$ is a selected cycle representative for the birth label $i$.

The theory of persistence pairing graphs is developed in the companion manuscript \cite{manzanares2026PPG}. Here, we use only the part needed for computation.

Fix a homological dimension $q$. A graph vertex corresponds to either a finite persistence interval of positive length or an essential persistence interval and keeps its birth and death values. Write $b(u)$ and $d(u)$ for the birth and death values of a vertex $u$, with $d(u)=+\infty$ for an essential interval. At an original filtration value $t$, define
\[
\Dying_t:=\{u:d(u)=t\},
\qquad
\Surv_t:=\{v:b(v)<t<d(v)\}.
\]
Thus $\Dying_t$ contains the labels that die at $t$, while $\Surv_t$ contains the labels already present before the event and still present after it.

Immediately before $t$, the selected active birth cycles form a homology basis. Each dying selected basis element can therefore be expressed, after the transition, in the basis formed by the surviving selected cycles. If
\[
[z_u]\longmapsto
\sum_{v\in\Surv_t} w_{uv}[z_v],
\]
then the graph contains a directed edge $u\to v$ exactly when $w_{uv}\ne0$. Over $\Ftwo$, each nonzero coefficient equals one.

A conceptually simple implementation of the persistence pairing graph performs three large tasks: building the cubical complex, reducing the boundary matrix, and extracting the graph. The algorithm is poor for large volumes because the number of cells grows according to \cref{eq:cubical-cell-count}, while a full sparse change-of-basis matrix can become much larger than the final barcode or graph.

There are two standard ideas in the persistent homology literature that already point toward a better design. First, cubical boundaries can be represented implicitly rather than by storing every boundary column \cite{wagner2012cubical,kaji2020cubicalripser}. Second, we can avoid storing representative information that can be recovered later; this principle appears in different forms in Ripser and U-match factorization \cite{bauer2021ripser,hang2021umatch}. As part of the graph computation, we adopt both ideas, but we must ensure that enough algebra is retained to recover transition relations.

The largest additional cost on quantized images comes from repeated scalar values. The total order \cref{eq:cell-order} refines a filtration value into many cellwise steps. Standard reduction can therefore create a pair $(i,j)$ with
\[
f(\sigma_i)=f(\sigma_j).
\]
This interval has length zero at the original scalar scale.

These pairs can dominate the cellwise algebra. In the DRP-372 cohort analyzed below, each $256^3$ scalar volume induces $135{,}005{,}697$ cubical cells before contraction, whereas the median compressed complex after cancellation within equal filtration values contains only $37{,}471$ generators. This large reduction motivates eliminating zero length pairs before the global persistence reduction.

\section{Efficient Implementation}
\label{sec:cancellation}

The mathematical core of the compression is an elementary cancellation in which the two removed generators have the same filtration value. The proof below is stated for a general filtered based chain complex and then applied to cubical data.

\subsection{Zero-Length Intervals}

\begin{lemma}
\label{lem:same value-cancellation}
Let $C_*$ be a finite filtered based chain complex over $\Ftwo$. Let $x\in C_q$ and $y\in C_{q+1}$ be basis elements with
\[
f(x)=f(y)=t,
\]
and suppose the coefficient of $x$ in $\partial y$ equals one. Then there exists a filtered based chain complex $C'_*$ obtained by removing $x$ and $y$ such that
\[
C_*\cong C'_*\oplus A_*
\]
by a filtration-preserving chain isomorphism, where $A_*$ is the two-generator chain complex
\[
0\longrightarrow \Ftwo y
\xrightarrow{\partial} \Ftwo x'
\longrightarrow0
\]
given $\partial y=x'$. Consequently, $C_*$ and $C'_*$ have naturally isomorphic persistence modules.
\end{lemma}

\begin{proof}
Suppose $\partial y=x+a$, where $x$ does not occur in $a$. We replace the basis element $x$ of $C_q$ by
\[
x':=\partial y=x+a.
\]
This change of basis is invertible because the coefficient of $x$ is one. This change of basis also preserves the filtration. For each $s$, let $F_sC_q$ denote the span of basis elements with filtration value at most $s$. Since every term of $a$ occurs in $\partial y$, it has filtration value at most $f(y)=t$. Hence, both
\[
x'=x+a \qquad\text{and}\qquad x=x'+a
\]
involve only elements of filtration value at most $t$. Thus, for $s < t$, neither $x$ nor $x'$ is present, while for $s\ge t$, replacing $x$ by $x'$ does not change the span. Therefore, every filtered subspace $F_sC_q$ is unchanged.

For the new basis, we have $\partial y=x'$. Consider any other basis element $z\in C_{q+1}$ whose boundary $\partial z$ contains $x'$. If we replace $z$ by $z+y$, then $x'$ is removed from $\partial z$. This change of basis also preserves the filtration. Since $x'$ occurs in $\partial z$, we have
\[
f(x')\le f(z).
\]
As $f(x')=f(y)=t$, it follows that $f(y)\le f(z)$. Thus $y$ is already present whenever $z$ is present, so replacing $z$ by $z+y$ does not change any filtered subspace. We can repeat the process for all such $z$. This ensures that $x'$ appears in the boundary of $y$ and in no other boundary of a $(q+1)$-chain. 

Next, we show that $y$ cannot appear in the boundary of a basis element $w\in C_{q+2}$. If $y$ appears in the boundary of $w \in C_{q+2}$, then the coefficient of $x'$ in $\partial^2w$ would be one, because $\partial y=x'$ and no other basis element of $C_{q+1}$ has a boundary containing $x'$. Hence, 
\[
\partial^2w \neq 0.
\]
This contradicts the property $\partial^2=0$.

We have now shown that the span of $y$ and $x'$ is a direct filtered subcomplex $A_*$. The subcomplex $A_*$ has trivial homology. Indeed, since
\[
\partial y=x'\neq 0,
\]
the element $y$ is not a cycle, while $x'$ is a cycle but also a boundary, namely $x'=\partial y$. Hence, $A_*$ has no nontrivial homology classes.

Since $A_*$ is isolated from the remaining generators, the transformed chain complex splits as
\[
C_*\cong C'_*\oplus A_*,
\]
where $C'_*$ is obtained by omitting the two generators $y$ and $x'$. Thus, $A_*$ may be removed as a contractible filtered summand.

For $s<t$, the summand $A_*$ is absent from the filtered complex, while for $s\ge t$, it is present but has trivial homology. Hence, at every filtration value,
\[
H_k(F_sC_*)\cong H_k(F_sC'_*).
\]
These isomorphisms are compatible with the inclusion maps between filtration levels, so $C_*$ and $C'_*$ define naturally isomorphic persistence modules.
\end{proof}

A \emph{value block} is the set of all basis cells having the same filtration value $t$. For each filtration value $t$, process all cells with $f(\sigma)=t$ in increasing dimension and, within each dimension, increasing grid index. 

When a cell $\sigma$ is processed, we generate its boundary $\partial\sigma$ and rewrite it using the replacement relations of cells that have already been cancelled. After these substitutions, the boundary contains only generators that remain in the compressed complex and uncancelled generators from the current value block.

If the largest remaining boundary element $x$ satisfies
\[
f(x)=f(\sigma)=t,
\]
then $x$ and $\sigma$ form a same value pair. By Lemma~\ref{lem:same value-cancellation}, this pair may be cancelled without changing the persistence module. The pair is removed, and a replacement relation for $x$ is stored so that any later occurrence of $x$ can be rewritten in terms of the remaining generators.

If the rewritten boundary is empty, or if its largest remaining element has filtration value strictly smaller than $t$, then no same value cancellation is possible. In this case, $\sigma$ is retained as a generator of the compressed complex together with its rewritten boundary. After all cells of value $t$ have been processed, the relations obtained within that value block are fixed and the algorithm proceeds to the next filtration value.

We use
\[
x\rightsquigarrow a
\]
as notation for a stored rewrite rule: after $x$ has been cancelled, every later occurrence of $x$ is replaced by the chain $a$. The symbol $\rightsquigarrow$ is algorithmic notation for substitution and does not denote equality in the original chain complex.

For example, suppose $f(a)=1$, $f(x)=f(y)=3$, and $\partial y=x+a$. When $y$ is processed, $x$ is the largest remaining boundary element and has the same filtration value as $y$. Hence, $x$ and $y$ are cancelled, and the rewrite rule $x\rightsquigarrow a$ is stored.

Now let $z$ be a later cell with
\[
\partial z=x+b.
\]
Since $x$ has already been cancelled, substituting the stored rule rewrites this boundary as
\[
a+b.
\]
Thus, later boundaries can be expressed entirely in terms of generators that remain in the compressed complex.

Because the projected boundary is stored in the same order used to process the current value block, its last (still-)uncancelled term of value $t$ is precisely the largest available same value row. The implementation therefore selects this term directly as the pivot, without separately searching all rows of the current block.

\begin{theorem}
\label{thm:block-preserves-persistence}
The final filtered chain complex obtained after all same value cancellations has a persistence module naturally isomorphic to that of the original filtered chain complex.
\end{theorem}

\begin{proof}
At each step, the algorithm cancels two basis elements of the same filtration value whose incidence coefficient is nonzero in the current transformed complex. By \cref{lem:same value-cancellation}, removing such a pair does not change the persistence module.

Since the chain complex is finite and each cancellation removes two generators, only finitely many cancellations occur. Composing the persistence-module isomorphisms obtained at each step gives a natural isomorphism between the persistence module of the original filtered chain complex and that of the final compressed complex.
\end{proof}

\begin{corollary}
In every homological dimension, the positive-length and essential persistence intervals of the compressed complex are exactly those of the original filtered complex, with the same multiplicities.
\end{corollary}

After all same value cancellations have been performed, every nonzero boundary term of a surviving generator has strictly smaller filtration value than that generator. This observation yields an exact count of the generators in the compressed complex.

For a linear map $T:U\to V$, its \emph{cokernel} is the quotient vector space
\[
\coker T:=V/\im T.
\]
For a filtration value $t$, let $i_t^q:H_q(C^{t^-})\longrightarrow H_q(C^t)$ be the inclusion induced map, and define
\[
b_q(t):=\dim\coker i_t^q,
\qquad
d_q(t):=\dim\ker i_t^q.
\]
Here, $b_q(t)$ counts, including multiplicity, dimension-$q$ births at $t$, while $d_q(t)$ counts dimension-$q$ deaths at $t$.

\begin{theorem}
\label{thm:generator-count}
Suppose a finite filtered based chain complex has the property that the boundary of every basis generator has strictly smaller filtration value. Let $c_q(t)$ be the number of dimension-$q$ basis generators with filtration value $t$. Then
\begin{equation}
\label{eq:relative-generator-count}
c_q(t)=b_q(t)+d_{q-1}(t).
\end{equation}
If the barcode has $N_{\mathrm{fin}}$ finite positive-length intervals and $N_{\mathrm{ess}}$ essential intervals, then the total number of basis generators is
\begin{equation}
\label{eq:two-fin-plus-ess}
\boxed{2N_{\mathrm{fin}}+N_{\mathrm{ess}}}.
\end{equation}
\end{theorem}

\begin{proof}
Consider the relative chain complex $C^t/C^{t^-}$. Since every boundary strictly lowers the filtration value, the boundary of a generator with filtration value $t$ lies entirely in $C^{t^-}$. Therefore the relative boundary map is zero. Hence
\[
H_q(C^t,C^{t^-})
\cong
C_q^t/C_q^{t^-},
\]
and therefore
\begin{equation}
\label{eq:relative-homology-dimension}
\dim H_q(C^t,C^{t^-})=c_q(t).
\end{equation}

The short exact sequence of chain complexes
\[
0\longrightarrow C^{t^-}
\longrightarrow C^t
\longrightarrow C^t/C^{t^-}
\longrightarrow0
\]
induces the standard long exact sequence in homology \cite[Section~2.1]{hatcher2002algebraic}. The part relevant here is
\[
H_q(C^{t^-})
\xrightarrow{i_t^q}
H_q(C^t)
\xrightarrow{j}
H_q(C^t,C^{t^-})
\xrightarrow{\delta}
H_{q-1}(C^{t^-})
\xrightarrow{i_t^{q-1}}
H_{q-1}(C^t).
\]
Exactness means
\[
\ker j=\im i_t^q,
\qquad
\ker\delta=\im j,
\qquad
\im\delta=\ker i_t^{q-1}.
\]
By the first isomorphism theorem,
\[
\im j
\cong
H_q(C^t)/\ker j
=
H_q(C^t)/\im i_t^q
=
\coker i_t^q.
\]
Also, $\delta$ maps $H_q(C^t,C^{t^-})$ onto $\ker i_t^{q-1}$ and has kernel $\im j$. Thus the displayed part of the long exact sequence yields the short exact sequence
\[
0
\longrightarrow
\coker i_t^q
\longrightarrow
H_q(C^t,C^{t^-})
\longrightarrow
\ker i_t^{q-1}
\longrightarrow
0.
\]
Taking dimensions over $\Ftwo$ and using \cref{eq:relative-homology-dimension} gives
\[
c_q(t)
=
\dim\coker i_t^q+\dim\ker i_t^{q-1}
=
b_q(t)+d_{q-1}(t),
\]
which proves \cref{eq:relative-generator-count}.

Finally, sum this identity over all homological dimensions $q$ and filtration values $t$:
\[
\sum_{q,t}c_q(t)
=
\sum_{q,t}b_q(t)
+
\sum_{q,t}d_q(t).
\]
Every barcode interval contributes exactly one birth, so
\[
\sum_{q,t}b_q(t)=N_{\mathrm{fin}}+N_{\mathrm{ess}}.
\]
Every finite interval contributes exactly one death and every essential interval contributes none, so
\[
\sum_{q,t}d_q(t)=N_{\mathrm{fin}}.
\]
Therefore
\[
\sum_{q,t}c_q(t)
=
(N_{\mathrm{fin}}+N_{\mathrm{ess}})
+
N_{\mathrm{fin}}
=
2N_{\mathrm{fin}}+N_{\mathrm{ess}},
\]
which is \cref{eq:two-fin-plus-ess}.
\end{proof}

The persistence barcode depends only on the persistence module, whereas the persistence pairing graph also depends on the representatives selected by the deterministic matrix reduction. Consequently, preservation of the persistence module alone does not guarantee preservation of the graph.

We therefore distinguish two correctness statements. First, the value block cancellations preserve the persistence module and hence the barcode. Second, once the cell order, same value pivot rule, and reduction convention are fixed, subsequent storage and streaming optimizations must reproduce exactly the same compressed chain complex and therefore the same reduction and graph structure. The next section establishes this stronger equivalence for the streaming construction.

\subsection{Compact Filtration Order}

The streaming constructor groups cells by
\[
\bigl(f(\sigma),\dim\sigma\bigr),
\]
and processes each group in increasing regular grid index. This agrees with the deterministic order in \cref{eq:cell-order}. The correctness argument requires only a deterministic ordering of the finitely many scalar values present in the input; it is not tied to a particular bit depth. In the present quantized application, the integer filtration values provide these ranks directly. Other scalar representations can be ordered by their distinct values before the same bucket construction is applied.

Suppose the regular grid indices in one bucket are
\[
x_0<x_1<\cdots<x_r.
\]
It is enough to retain the first index $x_0$ and the positive successive differences
\[
\Delta_k=x_k-x_{k-1},
\qquad 1\le k\le r.
\]
Because a bucket is consumed sequentially, its indices are reconstructed cumulatively from
\[
x_k=x_{k-1}+\Delta_k.
\]
The particular byte encoding of these positive integers is an implementation detail and is not needed for the equivalence argument below.

The original scalar array is needed only to assign filtration values and can be released before contraction begins. Thereafter, the boundary of each cell is generated on demand from its regular grid index and the grid strides. A $q$-dimensional cube has $2q$ codimension one faces, and hence at most $2d$ in ambient dimension $d$. Thus, for fixed $d$, boundary generation requires only a constant amount of local index computation per cell.

This follows the general strategy of implicit complex representations used in persistence computations: local incidence relations are regenerated when needed rather than storing the complete boundary matrix \cite{bauer2021ripser,wagner2012cubical,kaji2020cubicalripser}.

\subsection{Streaming Equivalence}

The streaming construction changes how the filtration order and boundaries are stored, but not the deterministic value block contraction described above. The following result shows that these storage changes leave the compressed complex unchanged.

\begin{theorem}
\label{thm:streaming-equivalence}
Consider a materialized and a streaming implementation of the same value block contraction over $\Ftwo$. Suppose that both implementations
\begin{enumerate}
\item process the cubical cells in the same deterministic order \cref{eq:cell-order};
\item generate the same boundary for each processed cell;
\item use the same replacement relation for every previously cancelled generator; and
\item choose the largest remaining same value boundary element as the cancellation pivot.
\end{enumerate}
Then, after each processed cell, the two implementations have the same surviving generators, the same cancellation relations, and the same rewritten boundary columns. Consequently, they produce identical compressed filtered chain complexes.
\end{theorem}

\begin{proof}
We proceed by induction over the common processing order of the cells. Before any cell is processed, both implementations have empty sets of surviving generators, cancellation relations, and compressed boundary columns, so their states agree.

Assume that the two states agree after the first $p-1$ cells have been processed, and let $\sigma$ be the $p$-th cell. By assumption~1, both implementations process the same cell $\sigma$, and by assumption~2 they generate the same initial boundary $\partial\sigma$.

This boundary may contain generators that were cancelled during earlier steps. By the induction hypothesis, every such generator has the same stored replacement relation in both implementations. Consequently, both perform the same substitutions and obtain the same rewritten boundary.

If this boundary contains an uncancelled element of the same filtration value as $\sigma$, assumption~4 selects the same largest such element $x$ in both implementations. Hence, both cancel the same pair $x,\sigma$. Since the rewritten boundary is identical, they also store the same replacement relation for $x$.

If no same value element remains, both implementations retain $\sigma$ and store the same rewritten boundary as its compressed boundary column.

Thus, in either case, the surviving generators, cancellation relations, and compressed boundary columns agree after processing the $p$-th cell. By induction, the two implementations agree after every processed cell and therefore produce identical compressed filtered chain complexes.
\end{proof}

\begin{corollary}
Any implementation change that affects only the storage representation preserves the deterministic value block output, provided that the four hypotheses of \cref{thm:streaming-equivalence} remain satisfied.
\end{corollary}

This includes replacing raw filtration-order arrays by consumable buckets, delta-coded order storage, compact integer indices and filtered-cell records, and flat storage for replacement relations.

\subsection{Efficient Graph Computation}

After the same value cancellations, every nonzero boundary term strictly lowers the filtration value, so the compressed complex contains no zero-length persistence pairs. Persistence reduction is therefore performed only on the remaining generators associated with positive-length and essential persistence. We next show how the additional relations required for the persistence pairing graph can be recovered from this single reduction, without performing a separate cohomology reduction.

The standard reduction records a change-of-basis matrix $V$ satisfying
\[
R=DV.
\]
Rather than storing $V$ explicitly, the compact implementation records only the column additions used during reduction. Whenever $R_j\leftarrow R_j+R_i$, where $i<j$, the source index $i$ is appended to the operation log of column $j$.

Each column addition corresponds to an elementary upper-unitriangular matrix $E$. Over $\Ftwo$, the operation is its own inverse, so $E^{-1}=E$. If the reduction performs the operations $E_1,\ldots,E_m$ in this order, then
\[
V=E_1E_2\cdots E_m,
\qquad\text{and}\qquad
V^{-1}=E_m\cdots E_2E_1.
\]
Thus, the action of $V$ or $V^{-1}$ can be reconstructed from the stored operation log without materializing the full matrix $V$.

\begin{proposition}
\label{prop:operation-log}
Let $x$ be a coordinate vector in the original cell basis. During a left-to-right reduction, the vector $V^{-1}x$ can be computed from the operation log without forming $V$ or $V^{-1}$. For each target column $j$ that belongs to the original support of $x$, apply in chronological order all source toggles stored for target $j$. A source index that is created by such a toggle does not need recursive processing.
\end{proposition}

\begin{proof}
Consider an elementary column operation, for $i < j$,
\[
R_j\leftarrow R_j+R_i.
\]
When the corresponding elementary matrix acts on a coordinate vector, it toggles coordinate $i$ exactly when coordinate $j$ is present. Thus, each operation can only create or remove a coordinate with smaller index than its target.

Because the reduction proceeds from left to right, all operations whose target is $i$ occur before any operation whose target is a larger index $j$. Hence, if processing the log of target $j$ toggles a new coordinate $i<j$, the operations associated with target $i$ have already occurred in chronological order. They must therefore not be applied recursively.

It remains to determine which target indices $j$ should have their logs replayed. Before the reduction reaches column $j$, no earlier operation can change coordinate $j$: such a change would require $j$ to be the source of an operation with some larger target, and that operation occurs only later. Therefore, coordinate $j$ is present when its operations are reached if and only if $j$ belonged to the original support of $x$.

Consequently, replaying, in chronological order, the stored source toggles for exactly the target indices in the original support of $x$ reproduces the action of $V^{-1}$, without constructing either $V$ or $V^{-1}$.
\end{proof}

For a finite persistence pair $(i,j)$, the reduced death column $R_j$ is a cycle. Indeed, since $R=DV$,
\[
R_j=DV_j,
\]
and therefore $DR_j=D^2V_j=0$. The implementation then computes
\begin{equation}
\label{eq:lambda-relation}
\lambda_j:=V^{-1}R_j,
\end{equation}
using the operation log described above. This expresses the reduced cycle in the selected reduction basis. Its leading active nonzero coordinate is the paired birth $i$; the remaining coordinates encode the representative-level relations used in the persistence pairing graph. The homological interpretation of these relations is developed in the companion theory paper \cite{manzanares2026PPG}.

A raw relation $\lambda_j$ may still contain birth labels corresponding to classes that died at an earlier filtration value. Before processing a death event at value $t$, each such label is eliminated using the relation stored when that class died. These substitutions are triangular: eliminating a dead birth label introduces only smaller birth labels in the fixed order. Hence, the process terminates. The resulting relation is expressed entirely in the selected basis of classes that remain active in $H_q(C^{t^-})$.

Suppose $m$ persistence intervals die at the same filtration value $t$. Immediately before $t$, order the active birth labels so that the $m$ labels dying at $t$ come first, followed by the labels that survive beyond $t$. In other words, we have a matrix of the form
\[
[\text{dying labels}\mid\text{surviving labels}].
\]
Express each of the $m$ projected kernel relations in these coordinates and place the resulting vectors as the rows of
\begin{equation}
\label{eq:event-matrix}
C_t=[C_D\mid C_S].
\end{equation}
The block $C_D$ records the coefficients of the dying labels, while $C_S$ records the coefficients of the surviving labels that occur in at least one of the relations.

\begin{lemma}
\label{lem:triangular-relations}
Let a projected kernel relation correspond to a dying birth label $i$. Its coefficient at $i$ is one, and every other active birth label occurring in the relation is smaller than $i$ in the fixed filtration-compatible order.
\end{lemma}

\begin{proof}
Before projection, the relation associated with the dying birth label $i$ has coefficient one at $i$, and every other nonzero coordinate has label smaller than $i$. Suppose that one such label $k<i$ died earlier. Its stored replacement relation contains $k$ together with only labels smaller than $k$. Eliminating $k$ therefore introduces only labels smaller than $k$, and hence smaller than $i$, while leaving the coefficient of $i$ unchanged. Repeating this substitution for every earlier-dead label yields a projected relation whose coefficient at $i$ is still one and whose remaining active labels are all smaller than $i$.
\end{proof}

\begin{theorem}
\label{thm:unit-lower}
Order the $m$ dying labels at value $t$ by increasing birth index
\[
i_1<i_2<\cdots<i_m.
\]
Then the dying block $C_D$ in \cref{eq:event-matrix} is lower triangular with every diagonal entry equal to one. In particular, $C_D$ is invertible over $\Ftwo$.
\end{theorem}

\begin{proof}
Row $r$ is the projected relation whose leading dying label is $i_r$. By \cref{lem:triangular-relations}, the coefficient of $i_r$ is one, while every other dying label appearing in that row has birth index smaller than $i_r$.

Since the columns are ordered
\[
i_1<i_2<\cdots<i_m,
\]
any such smaller dying label must occur in one of the columns $1,\ldots,r-1$. Thus all entries above the diagonal are zero, and every diagonal entry is one. Hence $C_D$ is unit lower triangular.

A triangular matrix over a field is invertible if and only if all of its diagonal entries are nonzero. Therefore, $C_D$ is invertible over $\Ftwo$.
\end{proof}

Each projected kernel relation maps to zero in $H_q(C^t)$. Let $d$ denote the vector of basis labels that die at $t$, and let $s$ denote the surviving basis labels. The $m$ relations collected in $C_t=[C_D\mid C_S]$ therefore satisfy
\[
C_Dd+C_Ss=0
\]
in $H_q(C^t)$. Since $C_D$ is invertible and subtraction equals addition over $\Ftwo$, we have $d=W_ts$ where
\begin{equation}
\label{eq:event-W}
W_t:=C_D^{-1}C_S.
\end{equation}
Thus, $(W_t)_{uv}=1$ means that the relation for dying label $u$ contains surviving label $v$, and the persistence pairing graph contains the directed edge $u\to v$.

\begin{corollary}
Since $C_D$ is unit lower triangular, $W_t$ can be obtained by solving
\[
C_DW_t=C_S
\]
by forward substitution without matrix inversion or pivot search.
\end{corollary}

A naive representation of $C_S$ would allocate one column for every persistence class surviving beyond $t$. For large graphs, this is unnecessary, because a given death event typically involves only a small subset of those survivors. We therefore restrict $C_S$ to the survivor labels that occur in at least one projected relation.

\begin{lemma}
\label{lem:support-restriction}
Let $v\in\Surv_t$ be a surviving label that does not occur in any projected dying relation at $t$. Then omitting the column corresponding to $v$ from $C_S$ does not change any other transition coefficient and removes no persistence pairing graph edge.
\end{lemma}

\begin{proof}
Since $v$ does not occur in any projected relation, its column in $C_S$ is zero. Hence, the corresponding column of
\[
W_t=C_D^{-1}C_S
\]
is also zero. It therefore contributes no graph edge. Moreover, multiplication by $C_D^{-1}$ acts independently on each column of $C_S$, so removing this zero column does not affect any of the remaining transition coefficients.
\end{proof}

The implementation therefore constructs each event system only on the \emph{touched survivors}, namely the survivor labels that occur in at least one projected relation, and applies the same forward substitution over $\Ftwo$ on this restricted system.

\section{Porous Media Permeability}
\label{sec:drp-methods}

\subsection{Image Preprocessing}

The DRP-372 dataset contains $217$ standardized three-dimensional digital porous-media samples together with transport simulations and geometric descriptors \cite{santos2022drp372}. The collection spans heterogeneous source projects and more than fifty material categories or lithologies, with binary segmentations standardized so that $0$ denotes pore and $1$ denotes solid. We use the $256^3$ branch because it provides a common computational domain and is large enough to stress the computational pipeline.

The public supplement contains $138$ expected $256^3$ samples. We retained samples for which both the binary pore geometry and the corresponding standard single-phase permeability output were available, yielding $119$ complete geometry and target pairs. Samples without this target were excluded, and alternative simulation outputs or volumes at higher resolution were not substituted.

Permeability was analyzed on the logarithmic scale while retaining the source convention for its reported units rather than assigning an unverified SI dimension. Across the $119$ samples, $y$ ranges from approximately $-16.36$ to $-10.82$, with mean $-12.67$. Porosity has Pearson and Spearman correlations of approximately $0.64$ and $0.65$, respectively, with $y$, providing a useful baseline for evaluating additional geometric and topological information.

Because DRP-372 contains related samples originating from common source records, material families, and source projects, validation was performed using provenance-aware groups rather than treating all samples as independent. Samples sharing the same original record were always kept within the same fold. The primary analysis used $37$ material-family groups in a deterministic five-fold split, with test-fold sizes $24,24,24,24,$ and $23$. As a stricter transfer experiment, models were also evaluated by leaving out each of the $18$ source projects in turn. The final cohort contained $109$ distinct exact-origin groups, and no exact-origin or family group was divided between training and test data.

For a binary volume, let $P$ be the pore voxels and $S$ the solid voxels. We compute the signed Euclidean distance transform
\begin{equation}
\label{eq:sedt}
s(x)=d(x,P)-d(x,S).
\end{equation}
With this convention, $s(x)<0$ in pore and $s(x)>0$ in solid. Thus, the sublevel filtration begins at deep pore centers, expands through pore bodies and constrictions, crosses the interface between pore and solid, and then penetrates the solid phase.

Distances are measured between voxel centers. We use voxel units rather than rescaling each source to physical length. This is deliberate: the application combines source projects with different or physical resolutions specified by users, while the target supplied here is the lattice unit system permeability output. A sample specific physical rescaling would therefore alter the common filtration scale in a way that is not matched by the response. No sample specific minimum to maximum normalization is applied.

The distance field is encoded by the common map
\begin{equation}
\label{eq:sedt-quantization}
q(x)=\operatorname{round}\!\left(32768+64\,s(x)\right),
\end{equation}
stored as an unsigned $16$-bit integer. Hence,
\[
s(x)=\frac{q(x)-32768}{64},
\qquad\text{and}\qquad
\ell=\frac{q_d-q_b}{64}
\]
for a finite persistence interval. One integer level represents $1/64=0.015625$ voxel and the rounding error is at most $0.0078125$ voxel. All $119$ samples were encoded without clipping. The cohort-wide signed distance range is approximately $-69.41$ to $260.74$ voxels.

\subsection{Feature Extraction}

Each quantized SEDT volume was processed using the streaming value block implementation described above. A $256^3$ scalar volume induces
\[
513^3=135{,}005{,}697
\]
cubical cells in the doubled-grid representation before same value contraction. For each specimen, we retained the persistence pairing graph together with persistence summaries and event-level information required for subsequent feature extraction.

All $119$ specimens completed successfully. Same value contraction substantially reduced the complexes: the median compressed complex contained $37{,}471$ generators, compared with $135{,}005{,}697$ original cubical cells, corresponding to a median compression factor of approximately $3{,}603$. Across the cohort, the compressed generator count ranged from $519$ to $257{,}987$. The median runtime was approximately $66.0$ seconds per specimen, and the median internally measured peak memory was $1020$ MiB, with a maximum of $1435$ MiB. The resulting graphs contained a median of $18{,}736$ vertices and $7{,}600$ directed edges.

The complete extraction stage produced $648$ numerical descriptors: $4$ geometry descriptors, $224$ persistence descriptors, $116$ pure graph descriptors, and $304$ coupled persistence and graph descriptors. For each homological dimension $q\in\{0,1,2\}$, the persistence core contains the number of finite intervals, median birth, median death, mean persistence, maximum persistence, total persistence, and normalized persistence entropy. If the finite persistence lengths are positive, let
\[
w_i=\frac{p_i}{\sum_{j=1}^n p_j}.
\]
The \emph{normalized persistence entropy}, for $n > 1$, is
\[
H_{\mathrm{pers}} = -\frac{\sum_{i=1}^n w_i\log w_i}{\log n}.
\]
It is set to zero when fewer than two finite intervals are present.

Let $G_q=(V_q,E_q)$ denote the persistence pairing graph in homological dimension $q$. Pure graph descriptors include edge count, edge to vertex ratio, source and sink fractions, isolated-vertex fraction, and largest weak component fraction. A vertex is isolated when both its in-degree and out-degree are zero. The largest weak component fraction is
\[
\frac{\max_{C\in\mathcal C_q}|C|}{|V_q|},
\]
where $\mathcal C_q$ denotes the connected components obtained after ignoring edge direction.

Coupled persistence and graph descriptors measure how persistence values are organized along graph relations. For finite vertices, we compute the Pearson correlations between persistence and in-degree and between persistence and out-degree. For finite to finite edges $(u,v)$, \emph{persistence assortativity} is
\[
\rho_q^{\mathrm{edge}} = \operatorname{corr}\bigl(p(u),p(v)\bigr),
\]
where $p(v)$ denotes the persistence of vertex $v$. We also record the mean absolute persistence difference across such edges,
\[
\frac{1}{|E_q^{\mathrm{ff}}|}\sum_{(u,v)\in E_q^{\mathrm{ff}}}|p(v)-p(u)|.
\]

Finally, the largest component persistence fraction measures how much of the total finite persistence is concentrated in a single weakly connected component. This descriptor is given by
\[
L_q^{\mathrm{pers}} = \frac{\displaystyle\max_{C\in\mathcal C_q}\sum_{\substack{v\in C\\v\ \mathrm{finite}}}p(v)}{\displaystyle\sum_{\substack{v\in V_q\\v\ \mathrm{finite}}}p(v)}.
\]
Undefined correlations or summaries caused by structural absence or zero variance are assigned the fixed value $0$.

\subsection{Model Training and Evaluation}

The primary estimator is Ridge regression \cite{hoerl1970ridge}. Predictor standardization is performed within the training data, and the regularization parameter is selected by nested cross-validation \cite{varma2006bias}. The outer evaluation uses the five predefined family grouped folds. Within each outer-training set, the Ridge penalty is selected by family grouped inner cross-validation over $25$ logarithmically spaced values from $10^{-6}$ to $10^6$. Both standardization and hyperparameter selection are therefore performed without access to the held out outer fold.

Predictive performance is evaluated from the pooled out-of-fold predictions using root mean squared error (RMSE), mean absolute error (MAE), $R^2$, and Spearman rank correlation. Comparisons between nested feature representations are made on the same held out observations. For an error metric $E$, improvement is defined as
\[
\Delta E = E_{\mathrm{baseline}} - E_{\mathrm{expanded}},
\]
so that positive values indicate lower prediction error after adding the new feature family. Uncertainty in these paired differences is estimated using $5000$ bootstrap replicates \cite{efron1993bootstrap} in which complete family groups, rather than individual specimens, are resampled.

As a robustness analysis, the same feature representations and outer folds are evaluated using partial least squares regression (PLSR) \cite{wold2001pls}. The predictors are standardized using the corresponding training data, and the number of latent components is selected by family grouped inner cross-validation. PLSR is used only as an alternative estimator; it is not used to redefine or select the feature sets.

Transfer to unseen source domains is evaluated separately using leave one project out (LOPO) validation. Each of the $18$ source projects is held out in turn, while hyperparameter selection within the remaining data is also grouped by project. Ridge and PLSR use the same feature representations as in the primary analysis. Because project sizes differ substantially, performance is summarized both by pooling all held out specimens and by first computing the error separately within each project and then averaging across projects with equal weight. Paired uncertainty is estimated by bootstrap resampling of complete projects.

LOPO is deliberately stricter than the family grouped analysis. The family grouped experiment evaluates prediction for unseen material families within the heterogeneous cohort, whereas LOPO asks whether relationships learned from the remaining source projects transfer to an entirely unseen project domain.

After the predictive analyses and feature sets were fixed, additional descriptive analyses were performed to investigate which PPG properties recur across models and why their predictive contribution varies between projects. These analyses were not used for feature selection, hyperparameter tuning, or refitting of the predictive models. We compare coefficient signs across the family grouped and LOPO Ridge and PLSR fits and examine whether project level shifts in PPG descriptors coincide with differences in the benefit obtained from adding PPG information.

Selected specimens are additionally examined spatially. For these analyses, the compressed generator indices are mapped back to their corresponding cells in the original doubled cubical grid. If $c$ is one coordinate in this grid, its location in voxel-index coordinates is
\[
x_{\mathrm{vox}} = \frac{c}{2}-\frac12.
\]
The resulting locations identify the critical cells at which selected persistence intervals are born and die under the chosen filtration and reduction convention. They should not be interpreted as the complete geometric support of the corresponding cycles or cocycles.

Finally, to distinguish graph connectivity from persistence magnitude alone, high persistence, high degree PPG vertices are compared with lower degree vertices from the same specimen and homological dimension having similar birth and death filtration values. Controls are required to lie within one voxel in both birth and death SEDT values. Distances between the birth cells, death cells, and interval midpoints of incident PPG neighbors are then used to examine the spatial organization associated with the graph relations.

\subsection{Permeability Prediction}
\label{sec:drp-results}

\paragraph{Prediction with Family Grouping.}
The family grouped results are summarized in \cref{tab:family-results,fig:family-rmse}. Persistence alone ($M2$) gives the largest prediction error among the main topological representations under both estimators. For Ridge regression, its pooled RMSE is $1.2660$. Augmenting persistence with pure PPG graph descriptors ($A2$) reduces the RMSE to $0.7857$, while adding the complete frozen PPG representation ($M3$) gives an RMSE of $0.8653$. The combined geometry, persistence, and PPG representation $M4$ gives the lowest Ridge error, with
\[
\mathrm{RMSE}=0.7665,\qquad R^2=0.4529,\qquad\text{and}\qquad\rho_{\mathrm{Spearman}}=0.8165.
\]
PLSR gives the same broad pattern. The combined $M4$ representation again has the lowest error,
\[
\mathrm{RMSE}=0.6585,\qquad R^2=0.5962,\qquad\text{and}\qquad\rho_{\mathrm{Spearman}}=0.8891.
\]
The agreement between Ridge and PLSR is useful because the two estimators handle the correlated feature sets differently.

\begin{table}[ht]
\centering
\caption{Pooled RMSE under family grouped held out validation. The response is $y=\log_{10}(k)$.}
\label{tab:family-results}
\begin{tabular}{@{}lrr@{}}
\toprule
Model & Ridge RMSE & PLSR RMSE\\
\midrule
$M0$ & 0.9056 & 0.8844\\
$M1$ & 0.9987 & 0.9750\\
$M2$ & 1.2660 & 1.3152\\
$A1$ & 0.9202 & 0.9088\\
$A2$ & 0.7857 & 0.8487\\
$M3$ & 0.8653 & 0.9215\\
$M4$ & \textbf{0.7665} & \textbf{0.6585}\\
\bottomrule
\end{tabular}
\end{table}

For Ridge, the paired $M2\rightarrow A2$ comparison reduces pooled RMSE by $0.4804$, with family cluster bootstrap interval
\[
[0.0198,0.9784],
\]
and improves prediction in all five outer folds. Thus, within the family grouped experiment, the pure graph descriptors provide additional held out predictive information beyond the selected persistence summaries. The broader $M2\rightarrow M3$ comparison reduces RMSE by $0.4007$, although its bootstrap interval $[-0.0722,0.9437]$ includes zero. PLSR gives corresponding point improvements of $0.4664$ and $0.3937$, respectively, with bootstrap intervals that also include zero.

The stronger conditional comparison asks whether PPG information remains useful once geometry and persistence are both already available. For Ridge, $A1\rightarrow M4$ reduces pooled RMSE by $0.1537$, but the bootstrap interval $[-0.0891,0.3770]$ includes zero. PLSR gives a larger point improvement of $0.2503$, again with an interval containing zero,
\[
[-0.0795,0.5524].
\]
The family grouped analysis therefore supports an incremental contribution of PPG information beyond persistence alone, whereas the evidence for an additional contribution beyond the combined geometry and persistence representation is weaker. Equal-weight family errors give the same broad ranking, with $M4$ having the lowest mean family RMSE for both Ridge ($0.4922$) and PLSR ($0.4397$).

\begin{figure}[ht]
\centering
\includegraphics[width=0.78\linewidth]{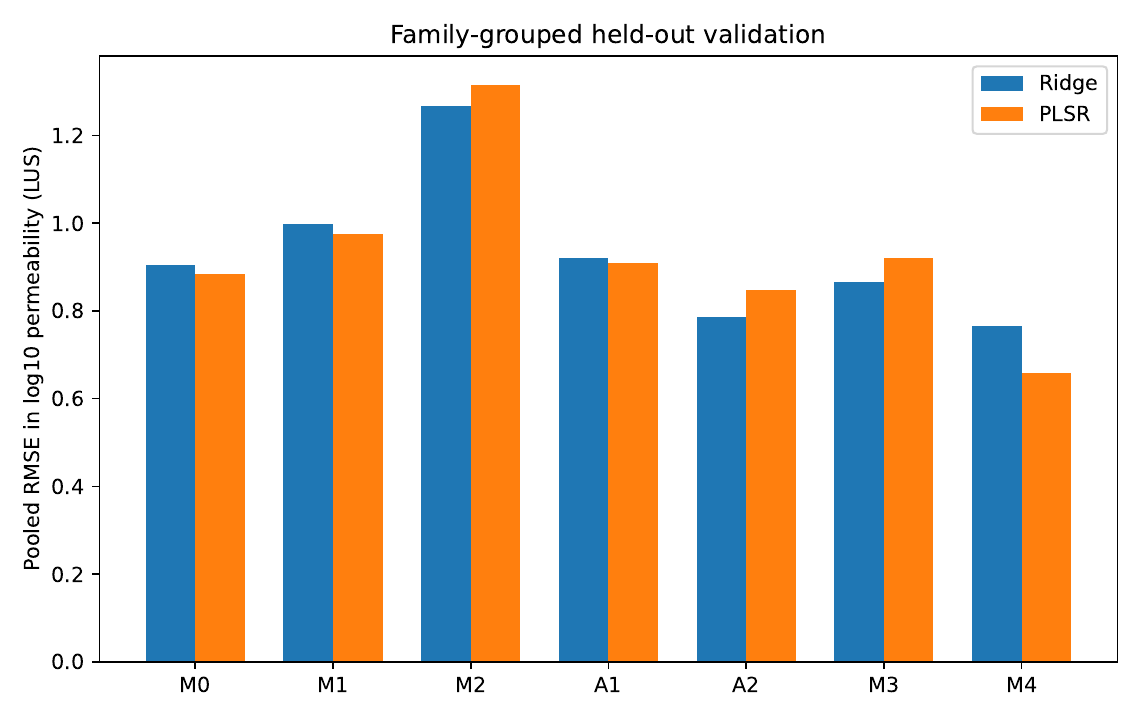}
\caption{Pooled RMSE under family grouped held out validation for Ridge and PLSR. The principal relational comparisons augment persistence alone ($M2$) with pure PPG structure ($A2$) or with the full frozen PPG representation ($M3$).}
\label{fig:family-rmse}
\end{figure}

\paragraph{Project Transfer Validation.}
The stricter leave one project out analysis gives a different ranking, as shown in \cref{tab:lopo-results,fig:lopo-rmse}. Under Ridge regression, geometry plus persistence ($A1$) gives the lowest pooled RMSE (0.6377) while the full $M4$ representation gives $0.6781$. Persistence alone ($M2$) is more competitive than in the family grouped experiment, with RMSE $0.9195$, and augmentation with the full PPG representation ($M3$) reduces this to $0.8288$.

\begin{table}[ht]
\centering
\caption{Pooled RMSE under strict leave one project out validation. The PLSR persistence only model $M2$ is strongly destabilized by the held out DRP374 project and is therefore not used as the principal effect-size comparison for PPG augmentation.}
\label{tab:lopo-results}
\begin{tabular}{@{}lrr@{}}
\toprule
Model & Ridge RMSE & PLSR RMSE\\
\midrule
$M0$ & 0.9358 & 0.9308\\
$M1$ & 0.9409 & 0.9810\\
$M2$ & 0.9195 & 3.5990\\
$A1$ & \textbf{0.6377} & \textbf{0.6604}\\
$A2$ & 0.8542 & 0.9729\\
$M3$ & 0.8288 & 1.0004\\
$M4$ & 0.6781 & 0.8268\\
\bottomrule
\end{tabular}
\end{table}

For Ridge, the paired $M2\rightarrow M3$ comparison reduces pooled RMSE by $0.0907$, with project cluster bootstrap interval
\[
[0.0137,0.2547].
\]
When projects are weighted equally, the mean RMSE improvement is $0.1006$, with interval
\[
[0.0070,0.1955],
\]
and $M3$ outperforms $M2$ in $12$ of the $18$ held out projects. The $M2\rightarrow A2$ comparison is also favorable on average, with pooled Ridge improvement $0.0653$, although its lower bootstrap bound lies close to zero. These results provide evidence that PPG relations can add transferable information beyond persistence summaries alone, although the magnitude of the improvement is considerably smaller than in the family grouped experiment.

\begin{figure}[ht]
\centering
\includegraphics[width=0.78\linewidth]{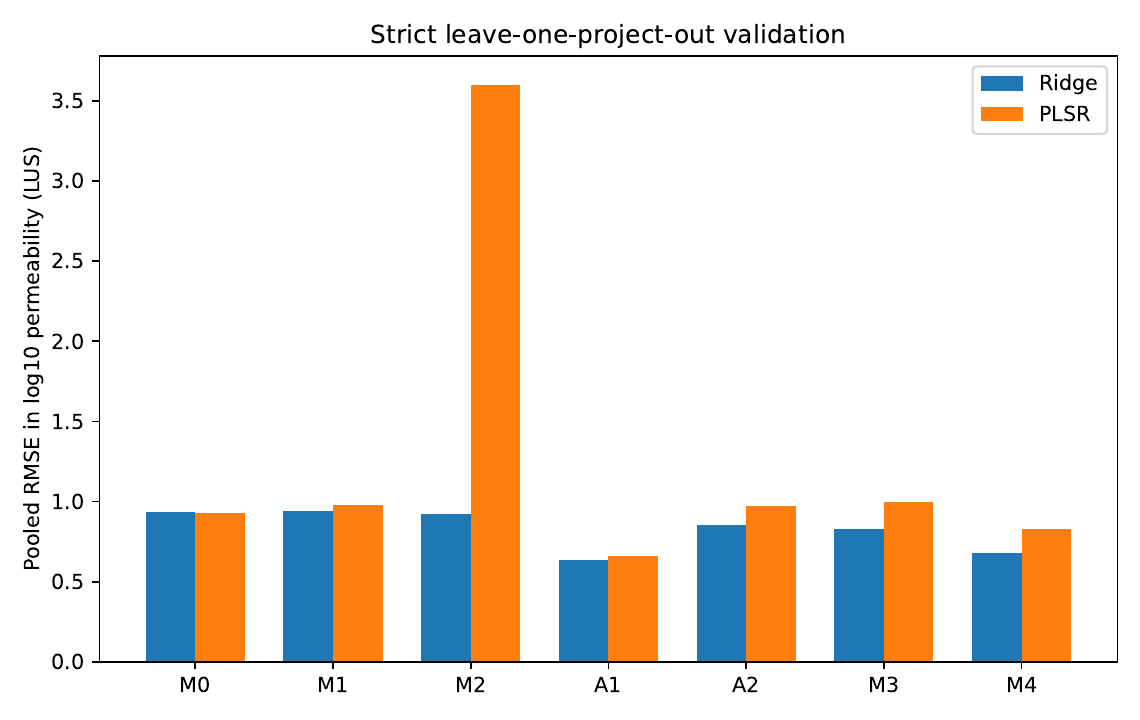}
\caption{Pooled RMSE under strict leave one project out validation. The PLSR persistence only model is strongly affected by DRP374; Ridge therefore provides the more stable comparison of $M2$ with its PPG-augmented representations.}
\label{fig:lopo-rmse}
\end{figure}

The geometry-conditional comparison is different. Ridge $A1\rightarrow M4$ has pooled improvement $-0.0405$ with interval
\[
[-0.1065,0.0855].
\]
At equal project weight, the difference is approximately neutral, with mean improvement $0.0092$ and interval $[-0.0542,0.0695]$; $9$ projects favor each representation. PLSR likewise does not show a consistent transfer advantage from adding PPG descriptors to geometry and persistence. We therefore do not infer that PPG augmentation generally improves prediction once both of those information sources are already available.

The PLSR $M2$ model is strongly destabilized when DRP374 is held out, and this project accounts for most of its squared prediction error. The resulting large PLSR improvements after graph augmentation are therefore not interpreted as reliable estimates of the incremental PPG contribution. Ridge provides the more stable transfer comparison.

\paragraph{Project Heterogeneity.}
The project level results further show that the value of PPG augmentation is domain dependent. The $M2\rightarrow M3$ comparison benefits both Ridge and PLSR in $11$ of the $18$ held out projects, harms both in $2$, and gives estimator-dependent results in $5$. For $M2\rightarrow A2$, the corresponding counts are $10$, $3$, and $5$. By contrast, the geometry-conditional $A1\rightarrow M4$ comparison benefits both estimators in $7$ projects, harms both in $8$, and is mixed in $3$.

A post hoc coefficient stability analysis identifies several recurring conditional associations: positive $H_1$ persistence--in-degree correlation, negative $H_0$ edge persistence assortativity, negative $H_2$ isolated-vertex fraction, positive $H_1$ edge persistence assortativity, and positive $H_1$ persistence--out-degree correlation. Because the predictors are correlated and these analyses were performed after the predictive comparisons were fixed, these patterns are interpreted descriptively rather than as independent or causal effects.

Several PPG descriptors also vary substantially between source projects. In the post hoc analysis, project membership accounts for approximately $57\%$ of the observed variance in $H_1$ edge persistence assortativity and approximately $46\%$ of the variance in $H_1$ edge count. DRP374 provides the clearest example of distribution shift: approximately $23.7\%$ of its pure PPG core values across samples and features lie outside the ranges observed in the remaining projects, and its mean $H_2$ isolated-vertex fraction is shifted by approximately $-3.85$ training standard deviations. DRP374 is also a project for which adding PPG information to geometry and persistence worsens both Ridge and PLSR prediction.

Across all projects, larger PPG feature range extrapolation tends to coincide with smaller $A1\rightarrow M4$ benefit. These associations do not survive correction for multiple comparisons, however, so they are treated as evidence of a possible role for distribution shift rather than as a demonstrated explanation of the observed domain dependence.

\subsection{Spatial Interpretation}
\label{sec:spatial}

The three representative specimens illustrate that the abundance of persistent features and their organization in the PPG need not vary together. Specimen \texttt{317\_01} contains $67{,}582$ PPG vertices and $20{,}643$ edges, including $30{,}833$ vertices in $H_2$. Nevertheless, its $H_2$ edge to vertex ratio is only approximately $0.093$, and about $87.4\%$ of its $H_2$ vertices are isolated. By contrast, \texttt{204\_011} contains only $411$ $H_2$ vertices, but has an $H_2$ edge to vertex ratio of approximately $0.219$ and an isolated-vertex fraction of about $66.2\%$. Thus a specimen may contain many persistence intervals while most of the corresponding PPG vertices participate in no graph relation. Persistence abundance and relational organization therefore provide distinct summaries of the filtration.

$H_1$ and $H_2$ vertices with high persistence and high degree were localized by mapping their compressed birth and death cells back to the doubled cubical grid. For all selected $H_1$ vertices in the three specimens, the birth value is negative and the death value is positive. These intervals are therefore created while the sublevel filtration is expanding through the pore phase and persist after the threshold has crossed the interface between pore and solid. The selected $H_2$ intervals, by contrast, have positive birth and death values. Under the convention in which pore SEDT values are negative, these features arise only after the filtration has entered the solid side of the signed distance field and should therefore not be interpreted automatically as cavities of the original pore space.

Persistence magnitude and graph prominence also differ substantially among the selected vertices. The selected $H_1$ vertices of \texttt{317\_01} have mean total PPG degree $181.2$, compared with $46.7$ for \texttt{374\_03\_011} and $12.1$ for \texttt{204\_011}. Among the selected $H_2$ vertices, \texttt{204\_011} has substantially greater mean persistence, approximately $18.33$ voxels, but mean degree only $2.9$; \texttt{317\_01} has mean persistence approximately $7.15$ voxels and mean degree $19.7$. These examples illustrate the distinction
\[
\text{persistence magnitude} \neq \text{relational prominence}.
\]
To examine whether graph degree can be explained simply by the local environment of a persistence critical cell, selected vertices were matched to lower degree controls from the same specimen and homological dimension whose birth and death SEDT values were each within one voxel of those of the selected vertex. Under this strict criterion, $21$ of the $60$ selected vertices admit a unique feasible match. Among these pairs, the median birth value difference is $0.172$ voxel and the median death value difference is zero.

The clearest matched comparison occurs for $H_2$ in \texttt{204\_011}, for which all ten selected vertices have interval matched controls. Despite a strong separation in PPG degree, the selected and control vertices have very similar local SEDT and pore fraction neighborhoods. Within radius $20$ voxels of the death cells, both groups have local porosity zero and deep-solid fraction one, with mean SEDT approximately $9.178$. The matched critical cells themselves are typically separated by more than $100$ voxels, so this similarity is not explained by repeatedly sampling the same spatial neighborhood.

For this matched $H_2$ example, PPG degree is therefore not explained by the immediate pore fraction or signed distance environment around the critical cell. This is a negative result in the sense that no clear local geometric surrogate for graph prominence is identified.

We next examined the spatial configuration of the actual PPG neighbors of the selected vertices (\cref{fig:neighbor-distances}). In $H_1$, incident graph relations commonly connect critical cells separated by tens of voxels. Across the three selected specimens, the mean birth-cell distance to an incident neighbor ranges from approximately $46$ to $61$ voxels, the corresponding mean death-cell distance from $64$ to $84$ voxels, and the mean interval-midpoint distance from $46$ to $66$ voxels. In \texttt{374\_03\_011}, approximately $62\%$ of incident relations associated with the selected vertices have midpoint separation exceeding $50$ voxels, while approximately $16\%$ exceed $100$ voxels.

The selected $H_2$ vertices show a different spatial pattern. In \texttt{204\_011}, incident $H_2$ neighbors have mean birth-cell separation only $2.47$ voxels but mean death-cell separation $26.10$ voxels. The corresponding median distances are approximately $2.35$ and $23.94$ voxels. These related intervals are also relatively close in birth filtration value while differing more strongly in death value and persistence. A similar, although weaker, contrast between birth and death separations is observed in the selected $H_2$ vertices of the other two specimens.

\begin{figure}[ht]
\centering
\includegraphics[width=0.88\linewidth]{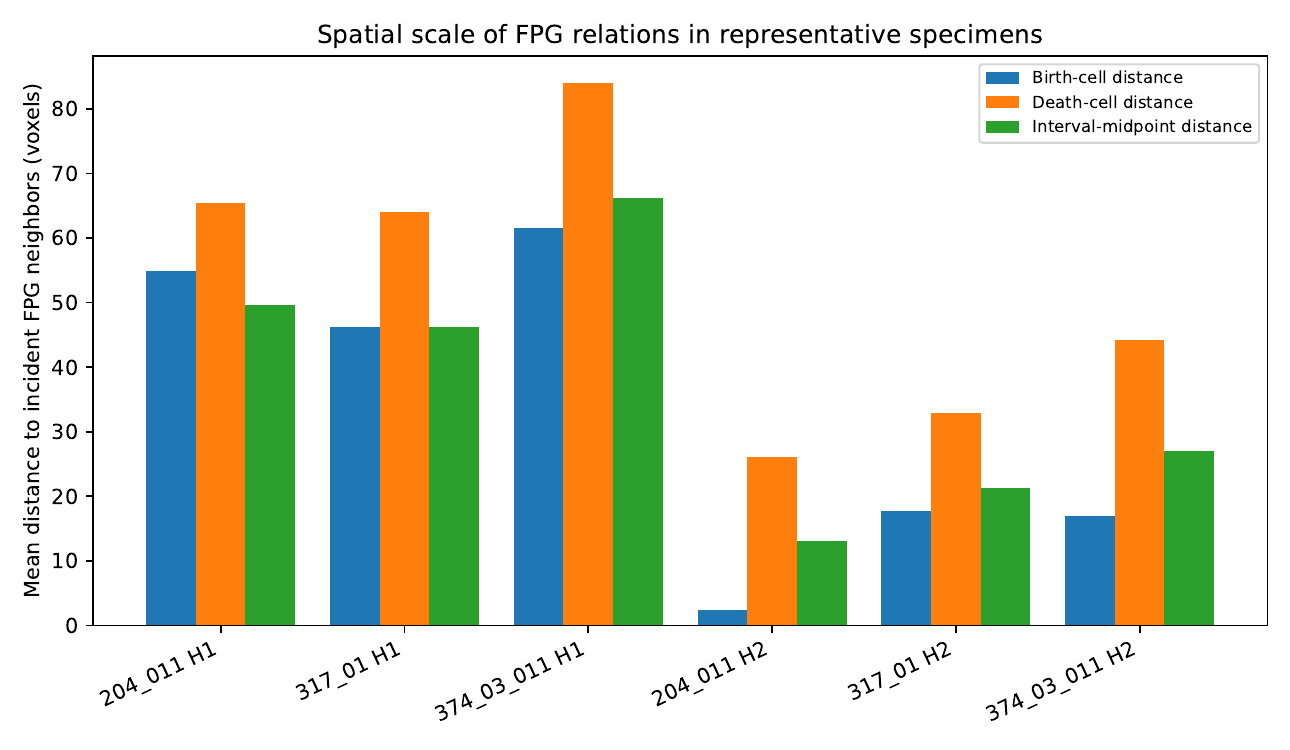}
\caption{Mean spatial distance from selected focal PPG vertices to their incident graph neighbors, measured between persistence birth cells, death cells, and midpoints between birth and death cells. PPG edges represent algebraic transition relations in the fixed reduction basis; these spatial separations should not be interpreted as physical pore channels.}
\label{fig:neighbor-distances}
\end{figure}

Taken together, the selected examples reveal two contrasting spatial patterns. The examined $H_1$ relations frequently connect persistence events whose critical cells are widely separated within the specimen. For the examined $H_2$ relations, birth cells can instead be spatially close while the corresponding death cells become substantially more separated. These observations concern the critical cells attached to algebraically related persistence intervals and are descriptive of the selected specimens. Recovering complete representative cycles or cocycles would be required to make stronger claims about the geometry or physical mechanism represented by individual PPG relations.

\paragraph{Permutation Null.}
To determine whether prediction depends on the specific placement of persistence intervals within the PPG, we performed a within-specimen permutation experiment. For each homological dimension, the PPG itself and the complete multiset of finite persistence intervals were held fixed while persistence intervals were randomly reassigned among finite graph vertices. Thus the null preserves both the barcode and the pure graph structure while destroying their observed coupling.

The observed coupling is strongly nonrandom. Persistence--in-degree correlations, edge persistence assortativity, edge persistence differences, and persistence concentration within graph components differ substantially from their permutation distributions in all three homological dimensions. For $12$ of the $13$ coupled core descriptors, the cohort mean statistic attains the smallest empirical two-sided $p$-value available from $500$ permutations ($p=1/501$).

This non-random organization does not, however, translate into a detectable improvement in family grouped prediction. For Ridge regression, persistence plus the observed coupled descriptors gives RMSE $1.0299$, compared with a permutation-null median of $0.9354$ ($p=0.721$ for a null assignment performing at least as well). When pure graph descriptors are also retained, the observed full PPG representation gives RMSE $0.8653$, compared with a null median of $0.8375$ ($p=0.711$). These results indicate that the observed persistence and graph coupling is highly structured, but the principal predictive gain over persistence in this experiment is associated with the pure PPG graph descriptors rather than with the particular assignment of persistence magnitudes to graph positions.

\paragraph{Event-Constrained Graph Null.}
To determine whether the predictive contribution of the pure PPG descriptors depends on the particular surviving classes selected by the observed graph relations, we constructed an event constrained target reassignment null. For every source vertex, its homological dimension, death event, and observed out-degree were held fixed. Each outgoing edge was then reassigned to a randomly selected vertex that was already born and remained alive after the source's death event. Consequently, the number of vertices, number of edges, edge to vertex ratio, source events, and source out-degrees were preserved, while target identity, in-degree organization, isolation, and weak component structure were allowed to vary.

The observed PPGs differ strongly from these surrogates. In $H_1$, for example, the mean observed isolated-vertex fraction is $0.618$, compared with $0.556$ under target reassignment, while the largest weak component fraction is $0.172$ compared with $0.241$. Similar differences occur in $H_2$, indicating that the observed graphs are more fragmented than graphs obtained by assigning the same outgoing relations to arbitrary eligible survivors.

This specific target organization is not required to reproduce the predictive gain of the pure graph representation. The observed $A2$ model has pooled Ridge RMSE of $0.7857$, whereas the $200$ event constrained surrogate representations have mean RMSE of $0.7571$ and $95\%$ empirical range of $[0.7375,0.7730]$. A surrogate performs at least as well as the observed representation in $199$ of the $200$ permutations ($p=200/201$). The corresponding mean surrogate Spearman correlation is $0.7831$, compared with $0.7179$ for the observed representation.

Thus the large improvement from persistence alone to the pure PPG representation cannot be attributed specifically to the identities of the surviving targets selected by the observed graph. Instead, it appears to be carried predominantly by graph information preserved under this null, including edge abundance and source side event organization. The surrogate graphs satisfy the filtration-time survivor constraint but are not asserted to be realizable PPGs arising from an actual persistence reduction.

\section{Discussion}
\label{sec:discussion}

The porous media application illustrates why retaining persistence pairing relations can be useful beyond computing the persistence barcode alone. The barcode records the birth and death of persistent classes, whereas the PPG additionally records, under the fixed reduction convention, how classes that die at a filtration transition are related to classes that remain active. In DRP-372, adding PPG summaries to persistence summaries improves held out permeability prediction in the family grouped analysis under both Ridge regression and PLSR, and a smaller improvement remains in the stricter leave one project out Ridge comparison.

The two null models make this result more specific. The persistence permutation null shows that persistence values are arranged on PPG vertices in a strongly nonrandom way: persistence and graph position are coupled far more strongly than expected under exchangeable assignment within a specimen and homological dimension. However, this exact assignment is not more predictive than its randomized counterparts. The event constrained target reassignment null reaches a similar conclusion from the graph side. The observed PPGs are more fragmented than surrogate graphs with the same edge counts, source death events, and source out degrees, yet the surrogate pure graph representations perform at least as well in the family grouped Ridge analysis. These surrogates satisfy the filtration survivor constraint but are not asserted to be realizable PPGs from an actual persistence reduction. Together, the null experiments indicate that the main predictive gain is associated with coarser graph and event information rather than with the exact persistence assignment or exact surviving targets of individual edges.

The transfer experiment further limits the strength of the predictive claim. Once geometry and persistence are both included, adding PPG descriptors does not produce a consistent improvement across held out projects. Several PPG descriptors also vary substantially between source projects, and DRP374 provides a clear example in which the graph representation differs strongly from the training domains and does not improve the geometry plus persistence model. The results therefore do not support a claim that PPG descriptors universally improve permeability prediction. Instead, they indicate that coarse PPG organization can contain predictive information beyond the selected persistence summaries while also carrying domain specific structure.

The spatial analyses provide a complementary interpretation. In the matched $H_2$ analysis for \texttt{204\_011}, large differences in PPG degree remain even when selected and control vertices have nearly identical birth and death filtration values and similar local signed distance environments. Within this matched example, graph prominence is therefore not explained by the immediate local pore fraction or SEDT environment around the critical cell. At the same time, the spatial locations of incident PPG neighbors show substantial organization: selected $H_1$ relations frequently connect critical cells separated by tens of voxels, while selected $H_2$ relations can associate intervals with nearby birth cells but more widely separated death cells.

These observations should not be given a direct transport interpretation. A PPG edge is an algebraic relation produced by the selected persistence reduction; it is not itself a fluid pathway, pore throat, or geometric connection. The present localization identifies the critical cells associated with the births and deaths of related persistence intervals, not the complete supports of the corresponding homology representatives. The spatial analysis therefore shows where algebraically related events occur, but stronger geometric or mechanistic interpretation requires recovery of the full representatives.

Previous work on porous media has demonstrated that persistent homology can characterize pore structure, percolation related scales, and multiphase configurations \cite{robins2016porous,jiang2018pore,herring2019trapping}. The present application addresses a complementary question: whether information is gained by retaining relations among persistent classes rather than summarizing their intervals independently. For the DRP-372 cohort, the answer is positive at the level of coarse graph organization but conditional at finer levels. The exact representative dependent organization is highly structured, yet the null models do not identify it as the unique source of predictive improvement.

A natural theoretical and computational direction is therefore to distinguish more clearly between representative dependent and more basis robust relational information. The companion theory paper separates persistence information, which is determined by the persistence module, from the representative dependent relations encoded by a fixed PPG construction \cite{manzanares2026PPG}. Empirically, robustness could be investigated by comparing graph summaries across controlled perturbations or deterministic reduction conventions. More invariant alternatives might also be sought at the level of transition subspaces, ranks, simultaneous death events, or other coarse event summaries rather than individual graph edges. On the application side, recovery of complete representative cycles or cocycles would permit direct comparison between selected persistence relations and geometric structures such as pore throats, loops, solid inclusions, or simulated flow fields.

The computational implementation and porous media analysis have several limitations. First, the optimized streaming implementation currently operates over $\Ftwo$. The same value cancellation argument extends to other fields whenever the selected pivot coefficient is invertible, but persistence reduction, relation projection, and event solves would require the corresponding field arithmetic.

Second, the present compact streaming implementation is optimized for a discretely ranked scalar filtration, as provided here by the quantized $16$ bit signed distance field. The mathematical ordering is not restricted to integer valued data, but unrestricted floating point input requires an additional deterministic ranking or ordering mechanism before the same compact bucket representation can be used.

Third, the PPG itself is representative dependent. The value block contraction preserves the persistence module and hence the barcode, but the individual graph relations are tied to the deterministic compression, ordering, and reduction convention used in this study. We therefore do not claim that each PPG edge is a canonical invariant of the persistence module or that individual edges are stable under arbitrary perturbations.

Fourth, DRP-372 is a heterogeneous observational collection rather than a controlled permeability experiment. Although validation groups related samples and the LOPO analysis holds out complete source projects, project sizes are imbalanced and several PPG descriptors contain substantial project specific variation. The results support predictive and structural associations but do not establish a causal relationship between any individual PPG descriptor and permeability.

Fifth, the permeability response is taken from the dataset's lattice unit system output, and the signed distance filtration is analyzed in voxel units on the standardized $256^3$ grids. The numerical relationships found here should therefore not be transferred directly to physical unit permeability measurements in datasets where voxel spacing, specimen dimensions, or simulation conventions differ.

Finally, the spatial analysis localizes persistence birth and death cells rather than complete representative cycles or cocycles. The observed spatial separations therefore concern critical cells attached to algebraically related intervals. Stronger geometric or mechanistic interpretations of individual PPG relations will require reconstruction and analysis of the corresponding representatives.

\section{Conclusion}

We developed a scalable method for computing persistence pairing graphs from large cubical filtrations and used it to examine whether graph organization among persistent classes contains information beyond the persistence intervals themselves. The computational pipeline removes same value pairs before global persistence reduction, constructs the resulting compressed complex in streaming form, and recovers PPG transitions from a single reduction through projected relations and support-restricted triangular event solves.

The method was applied successfully to all $119$ DRP-372 specimens. Each $256^3$ signed distance volume induces $135{,}005{,}697$ cubical cells before same value contraction, whereas the median compressed complex contains only $37{,}471$ generators. This reduction makes persistence and PPG extraction feasible at the scale required for the permeability study.

Under family grouped held out validation, augmenting persistence summaries with PPG information improves permeability prediction under both Ridge regression and PLSR. Under the stricter leave one project out analysis, Ridge continues to show a modest improvement when PPG information is added to persistence alone. This advantage is not systematic once geometry and persistence are already included, however, and several PPG descriptors vary substantially between source projects. The predictive contribution of the graph representation is therefore both incremental and domain dependent.

The null model analyses further refine where this contribution appears to reside. Randomly permuting persistence intervals among vertices while holding both the barcode and the graph fixed shows that the observed persistence and graph coupling is strongly nonrandom. Nevertheless, the observed coupling does not outperform randomized assignments in family grouped permeability prediction. Similarly, event constrained reassignment of edge targets produces surrogate graphs that differ substantially from the observed PPGs in isolation and component structure, yet these surrogates perform at least as well as the observed pure graph representation. Thus the predictive improvement over persistence alone cannot be attributed specifically to the exact persistence values assigned to graph positions or to the identities of the surviving targets selected by individual PPG edges. Instead, the useful signal appears to lie primarily in coarser graph and event organization preserved by these null models.

The spatial analyses provide a complementary distinction between persistence and graph organization. Specimens with many persistence intervals can nevertheless have sparse and highly isolated PPG structure, whereas specimens with fewer intervals can exhibit relatively denser relations. In an interval matched $H_2$ example, vertices with markedly different graph degree have similar local signed distance environments, while their incident graph relations exhibit distinct spatial organization. These observations concern critical cells associated with algebraically related persistence intervals and do not imply that individual PPG edges represent physical flow pathways.

Taken together, the computational, predictive, null-model, and spatial results support a more specific conclusion: coarse persistence pairing graph organization contains information beyond persistence summaries, but the predictive value is not attributable uniquely to the exact representative-dependent edge assignments and is not uniformly transferable across source domains. The next challenge is therefore not only to compute persistence pairing graphs efficiently, but to identify coarser and potentially more invariant relational summaries that retain their useful structural information while reducing sensitivity to representative choice and source domain. Recovering complete representative cycles or cocycles and relating them to pore geometry and simulated flow fields provides a complementary route toward stronger geometric and physical interpretation.

\section*{Code Availability}

The Rust implementation for computing persistence pairing graphs from cubical filtrations is available at \url{https://github.com/jhnrckmnznrs/persistence_pairing_graph}.

\bibliographystyle{plainnat}
\bibliography{references}

\end{document}